\documentclass[11pt]{amsart}

\usepackage[a4paper,margin=1in]{geometry}
\usepackage{amsmath,amssymb,amsthm,mathtools,bm}
\usepackage{stmaryrd}
\usepackage{microtype}
\usepackage{enumitem}
\usepackage{xcolor}
\usepackage{graphicx}
\usepackage{comment}
\newtheorem{lemma}{Lemma}[section]
\newtheorem{assume}{Assumption}[section]
\newtheorem{theorem}[lemma]{Theorem}
\theoremstyle{proposition}
\newtheorem{proposition}{Proposition}[section]
\newtheorem{corollary}[lemma]{Corollary}
\theoremstyle{remark}
\newtheorem{remark}[lemma]{Remark}

\newcommand{\R}{\mathbb{R}}
\newcommand{\Sd}{\mathbb{S}^{d}}

\newcommand{\Fh}{\mathcal{F}_h}
\newcommand{\Sh}{\Sigma_h}
\newcommand{\Uh}{U_h}

\newcommand{\diver}{\operatorname{div}}
\newcommand{\jump}[1]{\llbracket #1\rrbracket}
\newcommand{\norm}[1]{\lVert #1\rVert}
\newcommand{\absj}[1]{\lvert #1\rvert_J}

\newcommand{\ip}[2]{\bigl(#1,#2\bigr)}

\title{Symmetric Taylor--Hood elements for the linear stress gradient problem}
\date{}
\author {Jun Hu}
\address{School of Mathematical Sciences, Peking University, Beijing 100871, P.R. China.}
\address{ Chongqing Research Institute of Big Data, Chongqing 401333, People’s Republic of China}
\email{  hujun@math.pku.edu.cn }
\author {Ting Lin}
\address{School of Mathematical Sciences, Peking University, Beijing 100871, P.R. China.}
\email{ lintingsms@pku.edu.cn }

\author {Shudan Tian}
\address{School of Mathematics and Computational Science, Xiangtan University, Xiangtan 411105, P. R. China. }
\email{ shudan.tian@xtu.edu.cn }
\author{Min Zhang}
\address{Department of Mathematics, College of Science,
Beijing Forestry University, Beijing 100083, P.R. China.}
\email{zhangminD01@bjfu.edu.cn}
\begin{document}
\begin{abstract}
We develop mixed finite element methods for the linear stress-gradient 
elasticity problem based on symmetric Taylor--Hood elements. We establish 
the stability of the symmetric Taylor--Hood pair on simplicial meshes in 
both two and three dimensions, thereby resolving the stability  
left open by Brezzi, Fortin, and Marini in 1993. Combing the Nitsche's 
method, we further construct finite element schemes for linear stress problem in general boundary condition. Numerical experiments confirm the theoretical 
results.
\end{abstract}
\maketitle

\section{Introduction}

\qquad In this paper, we focus on the linear stress gradient model. Given a bounded domain $\Omega\subset \mathbb
 R^d,~d=2,3$, the linear stress gradient elasticity equation is 
\begin{equation}
	\begin{cases}
\epsilon(u)= \mathcal{A} ( \sigma-\iota^2\triangle\sigma ) & \text{in }\Omega,  \\
\diver \sigma =  f & \text{in }\Omega,  \\
 \sigma  n =  g_f~&\text{on }\Gamma_f,\\
 u =  g_c~&\text{on } \Gamma_c ,\\
\partial_n\sigma = g_{\sigma}&\text{on }\partial\Omega,
	\end{cases}
	\label{lsg-strong}
\end{equation}where $\sigma$, $\epsilon(u) =\frac{1}{2}(\nabla u + \nabla u^T)$, and $u$ represent the stress (symmetric matrix-valued), strain (symmetric matrix-valued), and displacement (vector-valued), respectively. $\mathcal A$ is an isotropic fourth-order tensor, defined as 
$$
 \mathcal{A}\sigma = \frac{1}{2\mu}\sigma-\frac{\lambda}{2\mu(2\mu+d\lambda)}\mathrm{tr}(\sigma)I.
$$ Clearly, $\mathcal A$ models the relationship between the stress and the modified strain.  In \eqref{lsg-strong}, \(\mu>0\), \(\lambda\geq 0\) , and
\(0<\iota\leq 1\). For $\iota=0$, \eqref{lsg-strong}  does not reduce exactly to the standard linear elasticity problem, since it contains one additional boundary condition $\partial_n\sigma = g_{\sigma}.$   
 $\Gamma_c$ and $\Gamma_f $ are a partition of $\partial \Omega$. Define
\[\Sigma_f := \{\sigma\in H^1(\Omega; \mathbb S^d)|~\sigma n=g_f|_{\Gamma_f}\} \text{ and } Q := L^2(\Omega; \mathbb R^d).\] 
Here $\mathbb S^d$ is symmetric $d$-dimensional matrix and $\mathbb R^d$ is a $d$-dimensional vector. When $\Gamma_f=\emptyset,~\Sigma_f$ becomes $H^1(\Omega;\mathbb S^d)$.  
The variational formulation of the linear stress gradient model \cite{VariationLSG} can be regarded as a perturbed Hellinger-Reissner elasticity variation.  The variational framework seeks a solution $ \sigma \in \Sigma_f,~ u \in Q$ such that
\begin{equation}
	\begin{cases}
	&\iota^2(\nabla  \mathcal{A} \sigma,\nabla  \tau)+( \mathcal{A}\sigma, \tau) + (\diver \tau, u) = F(g_c,g_{\sigma};\tau),~\forall  \tau \in \Sigma_0,\\
	&(\diver \sigma, v) = ( f, v),~\forall v\in Q.
	\end{cases}
\label{lsg-weak}
\end{equation} 
 $(\sigma,\tau):=\int_{\Omega}\sigma:\tau\mathrm{d} x$ is an inner product on $\Sigma_f.$ The space $\Sigma_0:=\{\sigma\in H^1(\Omega;\mathbb S^d)|~\sigma n =0|_{\Gamma_f}\}.$  The right-hand side is
\[
\begin{aligned}
F(g_c,g_\sigma;\tau)
:=
\langle g_c,\tau n\rangle_{\Gamma_c}
+
\iota^2
\langle \mathcal{A}^{-1}g_\sigma,\tau\rangle_{\partial \Omega}.
\end{aligned}
\] Comparing with standard Hellinger-Reissner principle, approximation problem \eqref{lsg-weak} is more complicated. Because the stress space require full $H^1$-continuity, and there no tangential component can be relax as bubble functions. 
 \subsection{Gradient elasticity model}
 Gradient elasticity model is widely used in nano materials. For such materials, classical linear elasticity cannot  capture size effects because it lacks an intrinsic length-scale parameter \cite{Aifantis2009,Aifantis2011overview}. Therefore, Aifantis et al. improved the strain-gradient model introduced by Mindlin in 1965 \cite{MINDLIN1964,MINDLIN1965} and proposed a simplified strain-gradient model involving only one additional length-scale parameter \cite{Aifantis1986,Aifantis1999}. This model leads to a fourth-order singularly perturbed problem, for which finite element methods have been developed recently \cite{Huang2025LSG,Ming2017, MingH2korn, MingMixLSG, HuangLSGSIAM}. 

Strain-gradient elasticity is one class of gradient elasticity and is complementary to stress-gradient elasticity \cite{VariationLSG}. By incorporating higher-order stress or strain terms into the constitutive relation, gradient-elasticity models can describe the size-dependent deformation of small-scale materials \cite{Aifantis2009,Aifantis2011overview, LamEtAl2003}. This paper is mainly concerned with the stress-gradient model, which was proposed by Eringen in 1983 \cite{Eringen1983}. A variational formulation of this model was presented in 2015 \cite{VariationLSG}.

\subsection{Compared with the linear elasticity model.} 
Compared with the classical Hellinger--Reissner formulation of linear
elasticity problem, the stress-gradient model contains an additional higher-order
perturbation term. And will degenerate to the classical Hellinger--Reissner formulation in the limiting case $\iota=0$.  

This is worthy to mention that the displacement equation of the stress-gradient
model can be decoupled from the stress equation, if $\Gamma_f=\emptyset$; see
\eqref{eq:z-weak}. But this representation is particularly
useful when the body force $ f$ is explicitly available and
requires $\triangle f$ is well defined. Nevertheless, the limiting linear-elastic stress generally does not satisfy the additional
higher-order boundary condition imposed by the stress-gradient model. Such a boundary incompatibility generates a stress boundary
layer, and consequently the stress predicted by the stress-gradient model
may differ from its linear-elastic counterpart even when the displacement
equation can be decoupled \cite{LTLSG}. 

When $\Gamma_f\neq\emptyset$, the traction boundary condition couples the stress
and displacement through the boundary terms in the mixed formulation, and
the above decoupling is no longer available. In this case, the difference
between the stress-gradient and classical linear-elasticity models may
become more pronounced. For this reason, we aim to develop a family of practical finite element methods for solving LSG problems. 

\subsection{Existing Finite Elements for Symmetric Elasticity.} 
There are relatively few studies on discretizations of the LSG problem. 
In \cite{LTLSG}, three families of CG--DG pairs were proposed in two dimensions. 
However, their theoretical stability requires the polynomial degree 
$k\geq 7$, which makes them impractical for computation. 
For classical linear elasticity, many practical finite elements are available; 
see \cite{{Arnold2002,Arnoldweaksym,Arnold3D,HuZhang14,HZ3D,HZfamily,HZlowest,GuzmanAfiled,PaulyZulehner23}}. 
Nevertheless, most of these elements are $H(\operatorname{div})$-conforming 
and therefore cannot be directly applied to the LSG problem, which requires 
an $H^1$-conforming approximation of the symmetric stress tensor. This naturally motivates the use of mixed finite element pairs with
continuous stresses for discretizing the LSG problem.

Brezzi, et. al, considered several
finite element pairs with continuous stress for linear
elasticity \cite{BrezziFortinMarini1993}. Some of these pairs are also conforming for the LSG problem,
including the MINI-type and symmetric Taylor--Hood pairs. The symmetric Taylor--Hood pair is particularly attractive, since it admits higher-order extensions and can be implemented using standard Lagrange finite elements. However,
\cite{BrezziFortinMarini1993} pointed out that the stability of the
symmetric Taylor--Hood pair had not been established. Although other discretizations are available for linear
elasticity, resolving this stability problem is particularly important for
the LSG model, for which practical $H^1$-conforming symmetric stress
approximations are much more limited.

\subsection{Symmetric Taylor--Hood element.} 
One main contribution in this article is prove the inf-sup stability for symmetric Taylor--Hood element. The Taylor--Hood element was originally introduced for the numerical
approximation of the incompressible Navier--Stokes equations
\cite{TaylorHood1973} and has since become a classical discretization
of the Stokes problem. On simplicial meshes, it employs continuous
piecewise polynomials of degree \(k\) for the velocity and degree
\(k-1\) for the pressure, with \(k\geq 2\). The inf--sup stability of
this family has been established in
\cite{BercovierPironneau1979,Verfurth1984,Boffi1994,Boffi1997}.
More recently, a local Fortin operator for the Taylor--Hood pair in arbitrary dimension was
constructed in \cite{GiraultScott2003,DieningStornTscherpel2022}.

The corresponding stability arguments do not extend directly to the
symmetric Taylor--Hood pair, since the standard Stokes constructions can not preserve the symmetric constraint. And this constraint may cause some coupling across the vertex patch when adapting the macroelement technique in \cite{Stenberg1990TaylorHood}. The construction of a symmetry-preserving Fortin
operator is similarly more involved.

Fortunately, for the elasticity problem considered here, the stress
space can be chosen without imposing any essential boundary conditions.
Consequently, boundary-edge bubbles remain admissible. This argument also relies
on the fact that, for every $d$-simplex $T$,
\[
    \mathbb S^d
    =\operatorname{span}\{t_e t_e^T:\ e\in\mathcal E(T)\},
\]
where $\mathbb S^d$ denotes the space of real $d\times d$ symmetric
matrices and $\mathcal E(T)$ is the set of edges of $T$. The resulting
edge-patch argument establishes the inf--sup stability of the
symmetric Taylor--Hood pair of $k\geq 2$ for the LSG problem. Although
the result is presented here only in two and three dimensions, the same proof can be extend to arbitrary dimensions. 

It is worth mentioning that CG--CG (continuous Galerkin--continuous
Galerkin) pairs do not suffer from the criss-cross singularities arising
for the CG--DG pairs studied in \cite{LTLSG}, since interelement
continuity automatically enforces the corresponding compatibility
condition at every criss-cross vertex.

When $k=1$, the displacement space reduces the piecewise
constant functions. In this case, stability can be recovered by adding
a stabilization term following the technique of
\cite{HZlowest}. We also study this stabilized formulation
and combine it with Nitsche's method to solve general boundary
conditions. The same boundary treatment can be similarly applied to the
symmetric Taylor--Hood discretization. In addition, we establish
higher-order regularity estimates for the stress $\sigma$ when
$\Gamma_f=\emptyset$.

\textbf{Organization.} The remainder of this paper is organized as follows. Section~2
establishes the well-posedness of the continuous problem and derives
the corresponding regularity estimates. Section~3 proves the stability
of the symmetric Taylor--Hood pairs. Section~4 develops a stabilized
mixed finite element method and take Nitsche's method for the
general boundary conditions. Finally, Section~5 presents
numerical experiments confirming the theoretical results.

\section{Well-posedness and regularity}
In this section, we will show the well-posedness and the $H^2$-regularity result of problem
\eqref{lsg-weak}. The detalied of well-posedness can be found in \cite{LTLSG}. 
Define the corresponding norm on $\Sigma_f$ 
\[
\|\tau\|_{\iota}^2
:=
\|\tau\|_{0,\Omega}^2
+
\iota^2|\tau|_{1,\Omega}^2
+
\|\diver \tau\|_{0,\Omega}^2.
\] 
We equip $u$ with the standard $L^2$ norm. Here $\|\cdot\|_{m,G}$ represents the standard $H^m(G)$ norm of the Sobolev space and we will omit $G$ when $G=\Omega.$ Here $m=0,1,\dots$ and $m=0$ represents the standard $L^2(G)$ norm. Define \[(u,v)_G=\int_G uv\mathrm{d}x,~\forall u,v\in L^2(G),~( u, v)_G=\int_G u\cdot v\mathrm{d}x,~\forall u,v\in L^2(G;\mathbb R^d).\]When $G=\Omega$, we will omit it.
\begin{theorem}
    For all $f\in L^2(\Omega;\mathbb R^2)$, there exists a unique solution pair $(\sigma_{\iota},u_{\iota})\in H^1(\Omega;\mathbb S^2)\times L^2(\Omega;\mathbb R^2)$ of \eqref{lsg-weak}, such that
    $$
   \|\sigma_{\iota}\|_{\iota}+\|u_{\iota}\|_{0}\lesssim \|f\|_{0}.
    $$
    Here the hidden constant is independent of $\iota$ and $\lambda$.
\end{theorem}
The asymptotic result of the solution can also be found in \cite{LTLSG}. Let $(\sigma_{\iota}, u_{\iota})$ solve \eqref{lsg-weak}, and $(u_0,\sigma_0)$ solve the following equation:
\begin{equation}
\begin{cases}
    (\mathcal A\sigma_0,\tau)+b(\tau,u_0)=0,& \forall\tau\in H(\diver;\Omega),\\
   b(\sigma_0,v)=(f,v), &\forall v\in L^2(\Omega).
\end{cases}
\label{eq:iota0}
\end{equation}


\begin{proposition}
\label{prop:bdlayer-type}
For \(0<\iota\le 1\), let $(\sigma_\iota,u_\iota)$ solve \eqref{lsg-weak}, and let $(\sigma_0,u_0)$ solve \eqref{eq:iota0}. 
Assume that $\sigma_0\in H^1(\Omega;\mathbb S^2)$. Then
\[
    \|\sigma_\iota-\sigma_0\|_\iota
    +
    \|u_\iota-u_0\|_{0}
    \lesssim
    \iota |\sigma_0|_{1}.
\]
If, in addition, $\sigma_0\in H^2(\Omega;\mathbb S^2)$, then
\[
    \iota|\sigma_\iota-\sigma_0|_{1}
    +
    \|\sigma_\iota-\sigma_0\|_{0}
    \lesssim
    \iota^{3/2}
    \|\partial_n\sigma_0\|_{0,\partial\Omega}
    +
    \iota^2|\sigma_0|_{2}.
\]
Consequently,
\[
    \|u_\iota-u_0\|_{0}
    \lesssim
    \iota^{3/2}
    \|\partial_n\sigma_0\|_{0,\partial\Omega}
    +
    \iota^2|\sigma_0|_{1}
    +
    \iota^2|\sigma_0|_{2}.
\]
\end{proposition}
To give the $H^2$ estimation for $\sigma_{\iota}$, we assume the elasticity problem solution has the following regularity result.
\begin{assume}
    Let $(\sigma_0,u_0)$ be the solution of the
classical elasticity problem \eqref{eq:iota0}, and
assume the following regularity estimate:
\begin{equation}
\label{eq:uniform-elasticity-regularity}
    \|\sigma_0\|_{2}
    +\|u_0\|_{1}
    \lesssim \|f\|_{1},
\end{equation}
where the hidden constant is independent of $\lambda$.
\label{Ass:regularity-elasticity}
\end{assume}
For the convex domain in 2D, assumption \ref{Ass:regularity-elasticity} is established in \cite{BrennerSung1992}.

\begin{theorem}
    \label{lem:pure-displacement-regularity}
Let $\Omega$ be a convex polygonal bounded domain. 
Let $\Gamma_f=\emptyset$, $\Gamma_c=\partial\Omega$, and assume that
$g_c=g_\sigma=0$. Let $0<\iota\leq 1$, $f\in H^1(\Omega;\mathbb R^d)$,
and suppose that the Lam\'e parameters satisfy
\[
    0<\mu_0\leq\mu\leq\mu_1,
    \qquad \lambda\geq 0.
\]
Assume $(\sigma_\iota,u_\iota)$ be the solution of the linear stress
gradient problem \eqref{lsg-strong} with homogenous boundary condition, and the corresponding elasticity solution $(\sigma_0,u_0)$ satisfies assumption \ref{Ass:regularity-elasticity}, then
\begin{equation}
\label{eq:gradient-stress-regularity}
    \|u_\iota\|_{1}
    +\|\sigma_\iota\|_{1}
    +\iota^{1/2}\|\sigma_\iota\|_{2}
    \lesssim \|f\|_{1}.
\end{equation}
\end{theorem}

\begin{proof}
Let $\mathcal C:=\mathcal A^{-1}$. The constitutive equation can be written as
\[
    \sigma_\iota-\iota^2\Delta\sigma_\iota
    =\mathcal C\epsilon(u_\iota).
\]
Taking the divergence and using $\mathrm{div}\sigma_\iota=f$, we obtain
\begin{equation} 
    \mathrm{div}\bigl(\mathcal C\epsilon(u_\iota)\bigr)
    =f-\iota^2\Delta f.
    \label{eq:uiota}
\end{equation}
Consequently, for $z_\iota:=u_\iota-u_0$,
\[
    \mathrm{div}\bigl(\mathcal C\epsilon(z_\iota)\bigr)
    =-\iota^2\Delta f
    \quad\text{in }\Omega,
    \quad
    z_\iota=0
    \quad\text{on }\partial\Omega.
\]
Its weak formulation reads
\begin{equation}
\label{eq:z-weak}
    2\mu(\epsilon(z_\iota),\epsilon(v))
    +\lambda(\mathrm{div} z_\iota,\mathrm{div} v)
    =
    \iota^2\langle\Delta f,v\rangle
\end{equation}
for all $v\in H_0^1(\Omega;\mathbb R^d)$. Taking $v=z_\iota$,
using Korn's inequality,  we find
\[
    \mu\|z_\iota\|_{1,\Omega}^2
    \lesssim
    \iota^2|f|_{1,\Omega}\|z_\iota\|_{1,\Omega}.
\]
It follows that
\begin{equation}
\label{eq:z-estimate}
    \|z_\iota\|_{1,\Omega}
    \lesssim \iota^2|f|_{1,\Omega}.
\end{equation}
We next prove a parameter-uniform estimate for the stress increment.
Since $z_\iota\in H_0^1(\Omega;\mathbb R^d)$, then $\int_{\Omega}\mathrm{tr}\epsilon(z_{\iota}) 
\mathrm{d}x=0.$ Thus, $\lambda\mathrm{tr}\epsilon(z_{\iota})\in L_0^2(\Omega)$. By the divergence inf-sup
condition of Stokes problem and \eqref{eq:z-weak},
\[
\begin{aligned}
    \|\lambda\mathrm{tr}\epsilon(z_{\iota})\|_{0}
    &\lesssim
    \sup_{0\neq v\in H_0^1(\Omega;\mathbb R^d)}
    \frac{(\lambda\mathrm{tr}\epsilon(z_{\iota}),\mathrm{div} v)}
         {\|v\|_{1}}      
    \lesssim
    \iota^2|f|_{1}.
\end{aligned}
\]
Therefore,
$
    \|\mathcal C\epsilon(z_\iota)\|_{0}
    \lesssim
    \iota^2|f|_{1,\Omega}.
$
Defining
$
    \tau_\iota
    :=\iota^{-2}\mathcal C\epsilon(z_\iota),
$
Then we have
\begin{equation}
\label{eq:q-estimate}
    \|\tau_\iota\|_{0}
    \lesssim |f|_{1}.
\end{equation}
Moreover, the constitutive equation becomes
\begin{equation}
\label{eq:sigma-reaction}
    \sigma_\iota-\iota^2\Delta\sigma_\iota
    =\sigma_0+\iota^2\tau_\iota,
    \quad
    \partial_n\sigma_\iota=0.
\end{equation}
We now decompose
\[
    \sigma_\iota=\sigma_0+B_\iota+R_\iota,
\]
where $B_{\iota}$ satisfies
\begin{equation}
\label{eq:boundary-layer-part}
\begin{cases}
    B_\iota-\iota^2\Delta B_\iota=0
        &\text{in }\Omega,\\
    \partial_nB_\iota=-\partial_n\sigma_0
        &\text{on }\partial\Omega,
\end{cases}
\end{equation}
and $R_{\iota}$ satisfies
\begin{equation}
\label{eq:regular-part}
\begin{cases}
    R_\iota-\iota^2\Delta R_\iota
       =\iota^2\bigl(\Delta\sigma_0+\tau_\iota\bigr)
        &\text{in }\Omega,\\
    \partial_nR_\iota=0
        &\text{on }\partial\Omega.
\end{cases}
\end{equation}
Set
$
    g:=\partial_n\sigma_0,
    ~
    \Phi_\iota:=\Delta\sigma_0+\tau_\iota.
$
Let
\[
    E_\iota(B_\iota)^2
    :=
    \|B_\iota\|_{0}^2
    +\iota^2\|\nabla B_\iota\|_{0}^2.
\]
Testing \eqref{eq:boundary-layer-part} with $B_\iota$ gives
\[
    E_\iota(B_\iota)^2
    =-\iota^2\langle g,B_\iota\rangle_{\partial\Omega}.
\]
The trace inequality implies
\[
\begin{aligned}
    \|B_\iota\|_{0,\partial\Omega}
    &\lesssim
    \|B_\iota\|_{0,\Omega}^{1/2}
    \|B_\iota\|_{1,\Omega}^{1/2} \lesssim
    \iota^{-1/2}E_\iota(B_\iota).
\end{aligned}
\]
Consequently,
\begin{equation} 
    E_\iota(B_\iota)
    \lesssim
\iota^{3/2}\|g\|_{0,\partial\Omega},~    \|B_\iota\|_{1}
    \lesssim
    \iota^{1/2}\|g\|_{0,\partial\Omega}.
\label{eq:B-energy}
\end{equation}

To derive the $H^2$ estimate, set
\[
    W_\iota:=\sigma_0+B_\iota.
\]
It follows from \eqref{eq:boundary-layer-part} that
\[
    \partial_nW_\iota=0,
    \qquad
    \Delta W_\iota
    =\Delta\sigma_0+\iota^{-2}B_\iota
    \in L^2(\Omega;\mathbb S^d).
\]
Applying the homogeneous Neumann regularity estimate
\cite{Neumannregularity1} componentwise, we obtain
\begin{align*}
    \|W_\iota\|_{2}
    &\lesssim
    \|\Delta W_\iota\|_{0}
    +\|W_\iota\|_{0} \\
    &\lesssim
    \|\sigma_0\|_{2}
    +\iota^{-2}\|B_\iota\|_{0}
    +\|B_\iota\|_{0} \\
    &\lesssim
    \|\sigma_0\|_{2}
    +\iota^{-1/2}\|g\|_{0,\partial\Omega}.
\end{align*}
Since the trace theorem gives
$
    \|g\|_{0,\partial\Omega}
    =\|\partial_n\sigma_0\|_{0,\partial\Omega}
    \lesssim
    \|\sigma_0\|_{2,\Omega},
$
and $0<\iota\leq1$, it follows that
\[
    \|W_\iota\|_{2,\Omega}
    \lesssim
    \iota^{-1/2}\|\sigma_0\|_{2,\Omega}.
\]
Therefore, 
\begin{equation}
\label{eq:B-H2}
    \|B_\iota\|_{2}
    \lesssim
    \iota^{-1/2}\|\sigma_0\|_{2}.
\end{equation}
Similarly, testing \eqref{eq:regular-part} with $R_\iota$ yields
$
    \|R_\iota\|_{0}
    +\iota\|\nabla R_\iota\|_{0}
    \lesssim
    \iota^2\|\Phi_\iota\|_{0}.
$
Note that $R_{\iota}$ satsifies
\[
    \Delta R_\iota
    =\iota^{-2}R_\iota-\Phi_\iota,
\]
then the homogeneous Neumann regularity estimate
\cite{Neumannregularity1} gives
\begin{equation}
\label{eq:R-H2}
    \|R_\iota\|_{2,\Omega}
    \lesssim
    \|\Phi_\iota\|_{0,\Omega}.
\end{equation}
By \eqref{eq:uniform-elasticity-regularity},
\eqref{eq:q-estimate}, and the edgewise trace theorem,
\[
    \|g\|_{0,\partial\Omega}
    +\|\Phi_\iota\|_{0}
    \lesssim
    \|\sigma_0\|_{2}
    +\|\tau_\iota\|_{0}
    \lesssim
    \|f\|_{1,\Omega}.
\]
Therefore, using \eqref{eq:B-energy} and the estimate for
$R_\iota$ in $H^1(\Omega)$,
\[
\begin{aligned}
    \|\sigma_\iota-\sigma_0\|_{1}
    &\leq
    \|B_\iota\|_{1}
    +\|R_\iota\|_{1} \\
    &\lesssim
    \iota^{1/2}\|g\|_{0,\partial\Omega}
    +\iota\|\Phi_\iota\|_{0} \lesssim
    \iota^{1/2}\|f\|_{1}.
\end{aligned}
\]
Moreover, by \eqref{eq:B-H2} and \eqref{eq:R-H2},
\[
\begin{aligned}
    \iota^{1/2}\|\sigma_\iota\|_{2}
    \lesssim
    \iota^{1/2}\|\sigma_0\|_{2}
    +\iota^{1/2}\|B_\iota\|_{2}
    +\iota^{1/2}\|R_\iota\|_{2} 
\lesssim
    \|f\|_{1}.
\end{aligned}
\]
Thus, we finish the proof.
\end{proof}
\begin{remark}
The problem \eqref{lsg-strong} may haven't the solution $(\sigma_\iota,u_{\iota})$ with compatible boundary condition $\partial_n\sigma_{\iota}=0$ and $u_{\iota}=0$. When such a solution does not exist, the finite element
solutions of the weak formulation \eqref{lsg-weak} converge to a solution satisfying the following boundary condition
\[
    \frac{1}{2}\left[
        u_\iota\otimes n
        +n\otimes u_\iota
    \right]
    +\iota^2\mathcal A
    \partial_n\sigma_\iota
    =0
    \quad\text{on }\partial\Omega .
\]
On a sufficiently smooth domain, combining the classical
$H^2$-regularity theory for elliptic systems
\cite{AgmonDouglisNirenberg1964} with the asymptotic estimate
\eqref{prop:bdlayer-type}, we obtain
\[
    \iota^2\|\sigma_\iota-\sigma_0\|_{2,\Omega}
    +\|u_\iota-u_0\|_{1,\Omega}
    \lesssim \iota^{3/2}\|f\|_{1,\Omega}.
\]
When \eqref{lsg-strong} admits a sufficiently regular solution, this
solution coincides with the solution of \eqref{lsg-weak} and
satisfies the estimate in
Theorem \ref{eq:gradient-stress-regularity}.
\end{remark}
\section{Symmetric Taylor--Hood formulation}
In this section, we will discrete problem \eqref{lsg-weak} by symmetric Taylor-hood pairs with natural homogenous boundary condition, i.e., $\Gamma_f=\emptyset,~g_c=0,g_\sigma=0$. Let $\mathcal T_h$ be a shape regular simplex triangulation of $\Omega,$ and the mesh size defined as $h$. And all $1$-dimensional face set denote as $\mathcal E_h.$ $\omega_e$ represents the element patch of $e,$ i.e,
$\omega_e=\{T\in \mathcal{T}_h| T\cap e=e\}$. In 2 dimensional, for interior edges, $\omega_e$ always include two elements and it is a star patch for $3$ dimension case. For boundary edges, $\omega_e$ only contain one element in 2D. For a piecewise polynomial function $v_h$ on $\mathcal{T}_h$, the jump of $v_h$ is defined as: For an interior face $F=T^+\cap T^-$, fix a unit normal $n_F$ from $T^+$ to $T^-$ and set
\[
  \jump{v_h}|_F:=v_h^+-v_h^-.
\]
On a boundary face, we set $\jump{v_h}:=v_h$. Define 
\begin{equation}
    \Sigma_k = \{\sigma\in H^1(\Omega;\mathbb S^d)|~\sigma|_T\in P_k(T;\mathbb S^d) \},~Q_{k-1}:=\{q\in H^1(\Omega;\mathbb R^d)|~q|_T\in P_{k-1}(T;\mathbb R^d)\},
\end{equation}
The corresponding discrete formulation of \eqref{lsg-weak} is: Seek $(\sigma_h,q_h)\in \Sigma_k\times Q_{k-1}$, such that
\begin{equation}
\begin{cases}
a^S(\sigma_h,\tau_h)+b(\tau_h,p_h)=(f,\mathrm{div}\tau_h),& \forall\tau_h\in \Sigma_k,\\
   b(\sigma_h,q_h)=(f,q_h), &\forall q_h\in Q_{k-1}.
\end{cases}
\label{lsg-THdiscrete}
\end{equation}
Here \[a^{S}(\sigma_h,\tau_h):= \iota^2(\nabla\mathcal A\sigma_h,\nabla\tau_h )+  (\mathcal A\sigma_h,\tau_h)+(\mathrm{div}\sigma_h,\mathrm{div\tau_h})\] 
 and $b(\sigma_h,q_h):=(\mathrm{div}\sigma_h,q_h).$ 
The associated norm of $\Sigma_k$ is defined as
\[
\|\sigma_h\|_{\iota}^2:=a_{\iota}(\sigma_h,\sigma_h)+\|\mathrm{div}\sigma_h\|_0^2,
\]
where $a_{\iota}(\sigma_h,\sigma_h):=\iota^2(\nabla\mathcal A\sigma_h,\nabla\sigma_h )+  (\mathcal A\sigma_h,\sigma_h).$
And $Q_k$ equip the standard $L^2$-norm.  The coercivity is naturally obtained the definition of $a^S(\cdot,\cdot)$. Next section we will prove the inf-sup condition. 
\subsection{Inf--sup stability of the symmetric Taylor--Hood pair}
This subsection establishes the discrete inf--sup stability of the
symmetric Taylor--Hood pair. For each edge
\(e=\overline{a_i a_j}\in\mathcal E_h\), let
\[
    t_e:=\frac{a_i-a_j}{|a_i-a_j|}
\]
be a unit tangential vector along \(e\). Let \(\phi_i\) denote the
continuous piecewise linear nodal basis function associated with the
vertex \(a_i\), and define the corresponding edge-bubble function by
$
    b_e:=\phi_i\phi_j,
$
then \(\operatorname{supp} b_e=\overline{\omega}_e\).
\cite{HuZhang14,HZ3D} shows that for every \(d\)-simplex \(T\), the rank-one tensors associated with its
edges satisfy 
\[
    \mathbb S^d
    =\operatorname{span}\{t_e t_e^{\top}:e\subset\partial T\},~d=2,3.
\]
A standard scaling argument, combined with the shape-regularity assumption, implies the following norm equivalent lemma. 
\begin{lemma}
\label{lem:edge-direction-equivalence}
For every fixed integer $m\geq0$, there exist constants $c_*,C_*>0$ such
that
\begin{equation}
 c_*\|\mathcal E\|_{0,T}^2
 \leq
 \sum_{e\subset T}\int_T b_e\bigl(t_e^T\mathcal Et_e\bigr)^2\,dx
 \leq
 C_*\|\mathcal E\|_{0,T}^2
 \qquad
 \forall \mathcal E\in P_m(T;\mathbb S^d),
\label{eq:uniform-edge-norm}
\end{equation}
for every $T\in\mathcal T_h$. The constants depend only on $m$ and the
shape-regularity parameter of $\{\mathcal T_h\}_h$, and are independent of
$T$ and $h_T$.
\end{lemma}
The key observation is that, for any \(q_h\in Q_{k-1}\), \(h_T\epsilon(q_h)|_T\) can be controlled by
tensor-valued test functions supported on the edge patches
\(\omega_e\), with \(e\subset\partial T\).
Figure~\ref{fig:overlapping-edge-patches} illustrates the overlapping
edge-patch construction in two dimensions. And (a) of Figure~\ref{fig:overlapping-edge-patches} shows the interior element. 
For a boundary element, as shown in (b) of Figure~\ref{fig:overlapping-edge-patches}, some of these patches are truncated by
\(\partial\Omega\) and may therefore be one-sided. The red, blue, and
yellow regions represent the three edge patches, while
\(t_1,t_2,t_3\) denote the corresponding unit tangential directions.
\begin{figure}[htbp]
  \centering
  \includegraphics[width=0.88\linewidth]{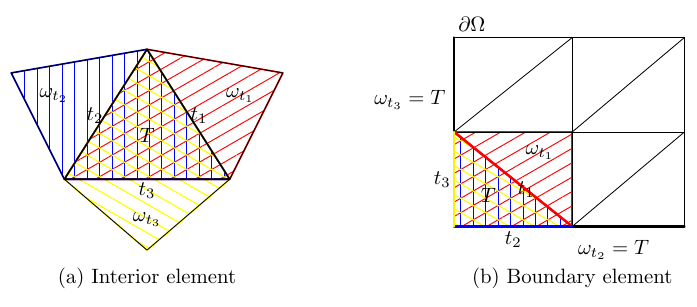}
  \caption{Overlapping edge patches associated with an interior element
(left) and a boundary element (right).}
  \label{fig:overlapping-edge-patches}
\end{figure}
\begin{lemma}[Edge-bubble control]
\label{lem:edge-bubble-strain-control}
There exists a constant $C>0$, depending only on $k$ and the
shape-regularity of $\mathcal T_h$, such that
\begin{equation}
 \|h\varepsilon(q_h)\|_{0}
 \leq C
 \sup_{0\neq\tau_h\in\Sigma_k}
 \frac{b(\tau_h,q_h)}{\|\tau_h\|_{1}}
 \qquad \forall q_h\in Q_{k-1}.
\label{eq:edge-bubble-strain-control}
\end{equation}
\end{lemma}
\begin{proof}
For any \(q_h\in Q_{k-1}\), let \(E_h:=\varepsilon(q_h)\). For each
\(e\in\mathcal E_h\), define a function \(\theta_e\) on \(\omega_e\) by
\[
    \theta_e|_T:=t_e^T E_h|_T t_e,
    \qquad T\subset\omega_e.
\]
We claim that \(\theta_e\in C(\omega_e)\).

In two dimensions, the continuity of \(q_h\) implies
\[
    \bigl(\nabla q_h|_{T^+}-\nabla q_h|_{T^-}\bigr)t_e=0
    ~\text{on }e.
\]
Since \(t_e^T\nabla q_h\,t_e=t_e^T E_h t_e\), it follows that
$
    \jump{\theta_e}|_e=0.
$
In three dimensions, note that each interior face of $\omega_e$ contains \(e\), and hence \(t_e\) is tangential vector on
\(F\). Therefore, the continuity of \(q_h\) gives
\[
    \bigl(\nabla q_h|_{T^+}-\nabla q_h|_{T^-}\bigr)t_e=0
    \qquad\text{on }F,
\]
which implies
\(\theta_e\in C(\omega_e)\).
Define
\begin{equation}
 \tau_e:=-h_e^2b_e\theta_e\,t_et_e^T,
 \qquad
 \tau_h:=\sum_{e\in\mathcal E_h}\tau_e,
\label{eq:def-global-edge-test}
\end{equation}
Noting that $\theta_e|_T\in P_{k-2}(T)$, thus $\tau_h\in\Sigma_h^k$.
By using the fact, the support of $\tau_e$ on $\omega_e$, yields
\begin{align}
 b(\tau_h,q_h)
 &=\sum_e\int_{\omega_e}\mathrm{div}\tau_e q_h\mathrm{d}x  \
 =\sum_{e\in\mathcal E_h}h_e^2
   \int_{\omega_e}b_e\theta_e^2\,dx \notag\\
 &=\sum_{T\in\mathcal T_h}\sum_{e\subset T}h_e^2
   \int_Tb_e(t_e^TE_ht_e)^2\,dx.
\label{eq:positive-edge-sum}
\end{align}
For any boundary $(d-1)$-dimensional face \(F\subset\partial\omega_e\cap\partial\Omega\),
we have \(t_e\cdot n_F=0\) if \(e\subset F\), whereas \(b_e|_F=0\)
otherwise.  Then Lemma~\ref{lem:edge-direction-equivalence} leads that
\begin{equation}
 b(\tau_h,v_h)
 \geq c\sum_{T\in\mathcal T_h}h_T^2\|E_h\|_{0,T}^2
 =c\|hE_h\|_{0}^2.
\label{eq:edge-test-lower}
\end{equation}
Since $b_e$ is non-negative on $T$, inverse estimate implies
\begin{align*}
 \|\tau_h\|_{1,T}\leq \sum_{e\subset T}h_e^2
       \|b_e\theta_e\|_{1,T}\leq Ch_T\|E_h\|_{0,T}.
\end{align*}
Consequently,
$
 \|\tau_h\|_{1}
 \leq C\|hE_h\|_{0}.
$
Thus, we finish the proof.
\end{proof}
\begin{lemma}[Verf\"urth estimate]
\label{lem:verfurth-symmetric-div}
There exists a constant $C>0$, independent of $h$, such that
\begin{equation}
 \|q_h\|_{0}
 \leq C\left(
 \sup_{0\neq\tau_h\in\Sigma_h^k}
 \frac{b(\tau_h,q_h)}{\|\tau_h\|_{1}}
 +\|h\varepsilon(q_h)\|_{0}
 \right)
 \qquad \forall q_h\in Q_{k-1}.
\label{eq:verfurth-symmetric-div}
\end{equation}
\end{lemma}
To prove Lemma \ref{lem:verfurth-symmetric-div}, we define the rigid body motion space
\[
 \mathrm{RM}:=
 \{r(x)=a+Bx:\ a\in\mathbb R^d,\ B^T=-B\}.
\]
For all $q_h\in Q_{k-1}$, we have the $L^2$-orthogonal decomposition
\begin{equation} 
q_h= r_h+q_0,~r_h\in  \mathrm{RM},~q_0\in \mathrm{RM}^{\perp}.
\label{eq:RM-decomp}
\end{equation}
The following Lemma is a direct consequence of the
negative-norm Korn inequality and the closed range theorem applied to
the symmetric divergence operator. \cite{PaulyZulehner23} gives the corresponding result in three dimensions.
\begin{lemma}
Let \(\Omega\) be a bounded connected Lipschitz domain. For every
\(q\in \mathrm{RM}^{\perp}\), there exists
\(\tau\in H_0^1(\Omega;\mathbb S^d)\) such that
\[
    \operatorname{div}\tau=q,
    \qquad
    \|\tau\|_{1}\lesssim \|q\|_{0}.
\]
\label{lem:symmetric-divtau}
\end{lemma}
\begin{proof}[Proof of Lemma \ref{lem:verfurth-symmetric-div}]
By Lemme \ref{lem:symmetric-divtau}, there exists $\tau_0\in H_0^1(\Omega;\mathbb S^d)$ such that
\[
\mathrm{div}\tau_0=q_0,~\|\tau_0\|_1\lesssim \|q_0\|_0,
\]
where $q_0$ is defined in \eqref{eq:RM-decomp}. Then let $I_h$ be the Scott-Zhang interpolation operator\cite{ScottZhang1990} from $H_0^1(\Omega;\mathbb S^d)\rightarrow H_0^1(\Omega;\mathbb S^d)\cap\Sigma_k $. Integration by parts implies 
\begin{align*}
 |(\operatorname{div}(\tau-I_h\tau),q_0)|
 &=|(\tau-I_h\tau,\varepsilon(q_0))|\\
 &\leq
 \|h^{-1}(\tau-I_h\tau)\|_{0}
 \|h\varepsilon(q_0)\|_{0}\lesssim\|\tau\|_{1}
 \|h\varepsilon(q_0)\|_{0}.
\end{align*}
For $r_h$, let
\[
 \rho_h:=
 \frac{a\otimes x+x\otimes a}{d+1}
 +
 \frac{(Bx)\otimes x+x\otimes(Bx)}{d+2}.
\]
A direct calculation gives
\[
 \rho_h\in P_2(\Omega;\mathbb S^d)\subset\Sigma_h^k,
 \quad
 \operatorname{div}\rho_h=r_h,
 \quad
 \|\rho_h\|_{1}\lesssim\|r_h\|_{0}.
\]
Let $\tau_h=I_h\tau_0+\rho_h,$ then
\begin{equation*}
    \frac{b(\tau_h,q_h)}{\|\tau_h\|_1}\gtrsim \frac{b(I_h\tau,q_0)+b(\rho_h,r_h)}{\|I_h\tau_0\|_1+\|\rho_h\|_1}\gtrsim\|q_h\|_0.
\end{equation*}
Thus, we finish the proof.
\end{proof}
Combing Lemma \ref{lem:edge-bubble-strain-control} and \ref{lem:verfurth-symmetric-div} we obtain the inf-sup stability. 
\begin{theorem}[Symmetric Taylor--Hood stability]
\label{thm:symmetric-taylor-hood}
For every $k\geq2$, there exists $\beta>0$, depending only on
$\Omega$, $k$, and the shape-regularity of $\mathcal{T}_h$, but independent of $h$, such
that
\begin{equation}
 \|q_h\|_{0}
 \leq \beta^{-1}
 \sup_{0\neq\tau_h\in\Sigma_h^k}
 \frac{(\operatorname{div}\tau_h,q_h)}
      {\|\tau_h\|_{1}}
 \qquad \forall q_h\in Q_{k-1}.
\label{eq:symmetric-th-infsup}
\end{equation}

\end{theorem}
\begin{theorem}[A priori error estimate]
\label{thm:symmetric-TH-error}
For $k\geq 2$, there exists an unique solution pair $(\sigma_h,u_h)\in \Sigma_k\times Q_{k-1}$ of problem \eqref{lsg-THdiscrete}. Let $(\sigma,u)\in\Sigma\times Q$ and
 be the solution of problem \eqref{lsg:s-weak}, then
\begin{equation}
 \|\sigma-\sigma_h\|_{\iota}
 +\|u-u_h\|_{0}
 \lesssim
 \inf_{\tau_h\in\Sigma_h^k}
 \|\sigma-\tau_h\|_{\iota}
 +
 \inf_{q_h\in Q_{k-1}}
 \|u-q_h\|_{0}.
\label{eq:symmetric-TH-quasi-optimal}
\end{equation}
Moreover, if
$
 \sigma\in H^{s+1}(\Omega;\mathbb S^d),
~u\in H^s(\Omega;\mathbb R^d),~ 1\leq s\leq k,
$
then
\begin{equation}
 \|\sigma-\sigma_h\|_{1}
 +\|u-u_h\|_{0}
 \lesssim
 h^s\left(
 |\sigma|_{s+1}
 +|u|_{s}
 \right).
\label{eq:symmetric-TH-optimal-error}
\end{equation}

\end{theorem}

\begin{remark}[Compatibility at a criss-cross vertex]
For the matrix-valued Falk--Neilan pair considered in
\cite{LTLSG}, a criss-cross vertex $z$ gives rise to the compatibility
condition
\begin{equation}
    J_z:\nabla q_1(z)+J_z:\nabla q_3(z)
    =
    J_z:\nabla q_2(z)+J_z:\nabla q_4(z),
    \label{eq:hm-cond}
\end{equation}
where $q_i:=q|_{T_i}$, and $J_{z,T}$ is a matrix depend on $T$, more detialed can be found in \cite{LTLSG}.

The condition \eqref{eq:hm-cond} is automatically satisfied
by the continuous Taylor--Hood displacement space. Indeed, let
$G_i:=\nabla q_i(z)$. Continuity of $q$ across the four edges meeting
at $z$ implies continuity of its tangential derivatives and hence
\[
\begin{aligned}
    (G_1-G_2)t_1&=0, & (G_3-G_4)t_1&=0,\\
    (G_2-G_3)t_2&=0, & (G_4-G_1)t_2&=0.
\end{aligned}
\]
It follows that
\[
    (G_1-G_2+G_3-G_4)t_1
    =
    (G_1-G_2+G_3-G_4)t_2=0.
\]
Since $t_1$ and $t_2$ are linearly independent,
\[
    G_1+G_3=G_2+G_4,
\]
which immediately yields \eqref{eq:hm-cond}. Thus, the compatibility
condition imposed on the Falk--Neilan pair at a criss-cross vertex
does not impose any additional restriction on the symmetric
CG--CG pair. 
\end{remark}

\section{Nitsche technical and the lowest order "Taylor--Hood" method}
In this section, we discretize the problem \eqref{lsg-weak} using a stabilized mixed method for the lowest order case, i.e., $k=1$ for $\Sigma_k$. In this case, $Q_{k-1}$ becomes piecewise constant space. Define 
\begin{equation*}
    \Sigma_h = \{\sigma\in H^1(\Omega;\mathbb S^d)|~\sigma|_T\in P_1(T;\mathbb S^d) \},~Q_{h}:=\{q\in H^1(\Omega;\mathbb R^d)|~q|_T\in P_{0}(T;\mathbb R^d)\},
\end{equation*}
 Moreover, we consider the discrete method for more general boundary conditions. Let $\Fh=\Fh^{i}\cup\Fh^{\partial}$ denote the set of all $(d-1)$-dimensional faces of $\mathcal{ T}_h$, where $\Fh^{i}$ and $\Fh^{\partial}$ denote the sets of interior and boundary faces, respectively. Let $\mathcal F_h^{\partial}= \Gamma_f\cup \Gamma_c,$ and we assume $\Gamma_c\neq \emptyset.$ Otherwise, there no displacement boundary condition to eliminate the rigid body motion. 
Note that $\sigma_h$ is not a subspace of $\Sigma_f$ if $\Gamma_f\neq \emptyset$. Therefore, the normal stress boundary condition \(\sigma n=g_f\) on \(\Gamma_f\) is imposed weakly by Nitsche's method \cite{Nitsche1971,Nitsche2012,Nitsche2021}. For notational simplicity, the symbols $\Gamma_f$ and $\Gamma_c$ are
also used for the corresponding unions of mesh boundary faces.
Define
\[
 a_\iota(\sigma,\tau)
 :=
 \iota^2(\nabla\mathcal A\sigma,\nabla\tau)
 +(\mathcal A\sigma,\tau),
\]
\[
 b_h^N(\tau,v)
 :=
 (\diver\tau,v)
 -\langle \tau n,v\rangle_{\Gamma_f},
\]
and the stabilized term is defined as 
\[
 c_h(w,v):=c_h^i(w,v)+c_h^c(w,v),
\]
where
\[
 c_h^i(w,v)
 :=
 \sum_{F\in\Fh^i}h_F
 \int_F\jump{w}\cdot\jump{v}\,\mathrm ds,
 \qquad
 c_h^c(w,v)
 :=
 \sum_{F\subset\Gamma_c}h_F
 \int_F w\cdot v\,\mathrm ds.
\]
The Nitsche's penalty is defined by
\[
 s_h^N(\sigma,\tau)
 :=
 \sum_{F\subset\Gamma_f}\frac{\gamma_f}{h_F}
 \int_F(\sigma n)\cdot(\tau n)\,\mathrm ds.
\]
The boundary functional is
\[
\begin{aligned}
 \mathcal F_h^N(g_f,g_c,g_\sigma;\tau_h,v_h)
 :={}&
 -\langle g_f,v_h\rangle_{\Gamma_f}
 +\langle g_c,\tau_h n\rangle_{\Gamma_c}
 +\iota^2\langle \mathcal A g_\sigma,\tau_h\rangle_{\partial\Omega}
 \\
 &+s_h^N(g_f,\tau_h)-c_h^c(g_c,v_h).
\end{aligned}
\]
The Nitsche-stabilized method seeks
$(\sigma_h,u_h)\in\Sh\times\Uh$ such that
\begin{equation}\label{lsg-nitsche}
\begin{cases} 
 &a_\iota(\sigma_h,\tau_h)+b_h^{N}(\tau_h,u_h)=\mathcal{F}_h^N(g_f,g_c,g_\sigma;\tau_h,v_h)
 \qquad\forall\tau_h\in\Sh,\\
 &b_h^N(\sigma_h,v_h)-c^N_h(u_h,v_h)+s_h^N(u_h,v_h)=\ip{f}{v_h}
 \qquad\forall v_h\in\Uh.
\end{cases} 
\end{equation}
Also we can denote it as
\begin{equation}
 B_h^N(\sigma_h,u_h;\tau_h,v_h)
 =
 (f,v_h)+\mathcal F_h^N(g_f,g_c,g_\sigma;\tau_h,v_h)
\end{equation}
for all $(\tau_h,v_h)\in\Sh\times\Uh$, where
\[
\begin{aligned}
 B_h^N(\sigma_h,u_h;\tau_h,v_h)
 :={}&
 a_\iota(\sigma_h,\tau_h)
 +b_h^N(\tau_h,u_h)
 +b_h^N(\sigma_h,v_h)
 \\
 &-c_h(u_h,v_h)+s_h^N(\sigma_h,\tau_h).
\end{aligned}
\]
The associated norms are
\[
 \|\tau_h\|_{\Sigma,\iota}^2
 :=
 a_\iota(\tau_h,\tau_h)
 +\|\diver\tau_h\|_0^2
 +s_h^N(\tau_h,\tau_h),
 \qquad
 \|v_h\|_J^2
 :=
 \|v_h\|_0^2+c_h(v_h,v_h).
\]

In particular, when $g_c=g_\sigma=0$ and $\Gamma_f=\emptyset$,
problem \eqref{lsg-nitsche} reduces to
\begin{equation}\label{lsg:s-weak}
\begin{cases}
 a_\iota(\sigma_h,\tau_h)+(\diver\tau_h,u_h)=0,
 &\forall\tau_h\in\Sh,\\
 (\diver\sigma_h,v_h)-c_h(u_h,v_h)=(f,v_h),
 &\forall v_h\in\Uh.
\end{cases}
\end{equation}
For $\iota=0$, this coincides with the stabilized mixed finite element
method for linear elasticity in \cite{HZlowest}. 
\subsection{Discrete inf--sup stability}
\begin{lemma}[Discrete inf--sup condition for the Nitsche scheme]
\label{lem:infsup-nitsche}
Assume that $\Gamma_c\neq\emptyset$. Then, for every
$(\sigma_h,u_h)\in\Sh\times\Uh$, there holds
\begin{equation}\label{eq:infsup-nitsche}
 \norm{\sigma_h}_{\Sigma,\iota}+\norm{u_h}_J
 \lesssim
 \sup_{(\tau_h,v_h)\in\Sh\times\Uh\setminus\{(0,0)\}}
 \frac{B^N_h(\sigma_h,u_h;\tau_h,v_h)}
      {\norm{\tau_h}_{\Sigma,\iota}+\norm{v_h}_J}.
\end{equation}
\end{lemma}

\begin{proof}
Let
\[
 \norm{\tau_h}_{a,\iota}^2:=a_\iota(\tau_h,\tau_h),
 \qquad
 |\tau_h|_{\Gamma_f,N}^2:=s_h^N(\tau_h,\tau_h),
 \qquad
 \absj{v_h}^2:=c_h(v_h,v_h).
\]
First, testing with $(\tau_h,v_h)=(\sigma_h,-u_h)$ gives
\begin{equation}\label{eq:basic-test-nitsche}
\begin{aligned}
 B_h^N(\sigma_h,u_h;\sigma_h,-u_h)
 &=
 a_\iota(\sigma_h,\sigma_h)
 +s_h^N(\sigma_h,\sigma_h)
 +c_h(u_h,u_h)  \\
 &=
 \norm{\sigma_h}_{a,\iota}^2
 +|\sigma_h|_{\Gamma_f,N}^2
 +\absj{u_h}^2 .
\end{aligned}
\end{equation}
Since $\Gamma_c\neq\emptyset$,
for every $u_h\in\Uh$ there exists
$\rho\in H^1(\Omega;\Sd)$ such that
\begin{equation}\label{eq:rho-lifting}
 \diver\rho=u_h,\qquad
 \rho n=0\quad\text{on }\Gamma_f,
 \qquad
 \norm{\rho}_{1}\lesssim \norm{u_h}_{0}.
\end{equation}
Then we want to control $\|u_h\|_0$. Let $\rho_h:=I_h\rho\in\Sh$, where $I
_h$ is the Scott-Zhang interpolation operator. By the stability of the interpolation
operator and the fact $\iota\leq 1$, we have

\begin{equation}\label{eq:rhoh-stability}
 \norm{\rho_h}_{\Sigma,\iota}
 \lesssim
 \norm{\rho}_{1}
 \lesssim
 \norm{u_h}_{0}.
\end{equation}
By the definition of $B_h^N(\cdot,\cdot)$,
\begin{align}
 b_h^N(\rho_h,u_h)
 &=
 (\diver\rho_h,u_h)
 -\langle \rho_h n,u_h\rangle_{\Gamma_f}
 \notag\\
 &=
 (\diver\rho,u_h)
 +(\diver(\rho_h-\rho),u_h)
 -\langle (\rho_h-\rho)n,u_h\rangle_{\Gamma_f},
\end{align}
where we used $\rho n=0$ on $\Gamma_f$. Since $u_h$ is piecewise
constant, elementwise integration by parts yields
\begin{equation*}
\begin{aligned}
 b_h^N(\rho_h,u_h)
 &=
 \norm{u_h}_{0}^2
 +\sum_{F\in\Fh^i}
   \int_F ((\rho_h-\rho)n_F)\cdot\jump{u_h}\,\mathrm ds  \\
 &\quad
 +\sum_{F\in\Gamma_c}
   \int_F ((\rho_h-\rho)n)\cdot u_h\,\mathrm ds .
\end{aligned}
\end{equation*}
By the Cauchy--Schwarz inequality and
\eqref{eq:rho-lifting},
\begin{align}
&\left|
\sum_{F\in\Fh^i}
   \int_F ((\rho_h-\rho)n_F)\cdot\jump{u_h}\,\mathrm ds
+\sum_{F\in\Gamma_c}
   \int_F ((\rho_h-\rho)n)\cdot u_h\,\mathrm ds
\right|
\notag\\
&\qquad
\le
\left(
 \sum_{F\in\Fh^i\cup\Gamma_c}
 h_F^{-1}
 \norm{(\rho_h-\rho)n_F}_{0,F}^2
\right)^{1/2}
\absj{u_h}
\notag\\
&\qquad
\lesssim
 \norm{\rho}_{1}\absj{u_h}
\lesssim
 \norm{u_h}_{0}\absj{u_h}.
\end{align}
Therefore, Young's inequality gives
\begin{equation}\label{eq:bhN-lift}
 b_h^N(\rho_h,u_h)
 \ge
 \frac34\norm{u_h}_{0}^2
 -C_J\absj{u_h}^2 .
\end{equation}
Moreover, by \eqref{eq:rhoh-stability},
\begin{align}
 a_\iota(\sigma_h,\rho_h)
 +s_h^N(\sigma_h,\rho_h)
 &\ge
 -\left(
 \norm{\sigma_h}_{a,\iota}
 +|\sigma_h|_{\Gamma_f,N}
 \right)
 \left(
 \norm{\rho_h}_{a,\iota}
 +|\rho_h|_{\Gamma_f,N}
 \right)
 \notag\\
 &\ge
 -\frac14\norm{u_h}_{0}^2
 -C_A\left(
 \norm{\sigma_h}_{a,\iota}^2
 +|\sigma_h|_{\Gamma_f,N}^2
 \right).
\end{align}
Combining this estimate with \eqref{eq:bhN-lift}, we obtain
\begin{equation}\label{eq:lift-test-nitsche}
\begin{aligned}
 B_h^N(\sigma_h,u_h;\rho_h,0)
 &=
 a_\iota(\sigma_h,\rho_h)
 +s_h^N(\sigma_h,\rho_h)
 +b_h^N(\rho_h,u_h)
 \\
 &\ge
 \frac12\norm{u_h}_{0}^2
 -C_A\left(
 \norm{\sigma_h}_{a,\iota}^2
 +|\sigma_h|_{\Gamma_f,N}^2
 \right)
 -C_J\absj{u_h}^2 .
\end{aligned}
\end{equation}
Next we want to control $\norm{\diver\sigma_h}_{0}$. Set
\[
 q_h:=\diver\sigma_h\in\Uh.
\]
Then
\begin{align}
 B_h^N(\sigma_h,u_h;0,q_h)
 &=
 b_h^N(\sigma_h,q_h)-c_h(u_h,q_h)
 \notag\\
 &=
 \norm{\diver\sigma_h}_{0}^2
 -\langle \sigma_h n,q_h\rangle_{\Gamma_f}
 -c_h(u_h,q_h).
\end{align}
The boundary term on $\Gamma_f$ is controlled by $s_h^N(\cdot,\cdot)$:
\begin{align}
 \left|\langle \sigma_h n,q_h\rangle_{\Gamma_f}\right|
 &\le
 \left(
 \sum_{F\subset\Gamma_f}
 \frac{\gamma_f}{h_F}
 \norm{\sigma_h n}_{0,F}^2
 \right)^{1/2}
 \left(
 \sum_{F\subset\Gamma_f}
 \frac{h_F}{\gamma_f}
 \norm{q_h}_{0,F}^2
 \right)^{1/2}
 \notag\\
 &\lesssim
 |\sigma_h|_{\Gamma_f,N}
 \norm{q_h}_{0}.
\end{align}
Here we used the inverse trace inequality. Also, since $q_h\in\Uh$, we have
\[
 \norm{q_h}_J\lesssim \norm{q_h}_0.
\]
Hence
\[
 |c_h(u_h,q_h)|
 \le
 \absj{u_h}\absj{q_h}
 \lesssim
 \absj{u_h}\norm{q_h}_{0}.
\]
Therefore, by Young's inequality,
\begin{equation}\label{eq:div-test-nitsche}
 B_h^N(\sigma_h,u_h;0,q_h)
 \ge
 \frac12\norm{\diver\sigma_h}_{0}^2
 -C_N|\sigma_h|_{\Gamma_f,N}^2
 -C_q\absj{u_h}^2 .
\end{equation}

Choose positive constants $\delta$ and $\eta$, independent of $h$ and
$\iota$, such that
\[
 \delta C_A\le \frac14,
 \qquad
 \delta C_J+\eta C_q\le \frac14,
 \qquad
 \delta C_A+\eta C_N\le \frac14.
\]
Define
\begin{equation}\label{eq:test-pair-nitsche}
 \tau_h^*:=\sigma_h+\delta\rho_h,
 \qquad
 v_h^*:=-u_h+\eta q_h .
\end{equation}
Combining \eqref{eq:basic-test-nitsche},
\eqref{eq:lift-test-nitsche}, and \eqref{eq:div-test-nitsche}, we get
\begin{align}
 B_h^N(\sigma_h,u_h;\tau_h^*,v_h^*)
 &\ge
 \beta\Big(
 \norm{\sigma_h}_{a,\iota}^2
 +|\sigma_h|_{\Gamma_f,N}^2
 +\norm{\diver\sigma_h}_{0}^2
 +\norm{u_h}_{0}^2
 +\absj{u_h}^2
 \Big)
 \notag\\
 &=
 \beta\Big(
 \norm{\sigma_h}_{\Sigma,\iota}^2
 +\norm{u_h}_J^2
 \Big).
\end{align}
Here $\beta>0$ is a constant depending on $\delta,~\eta$. Finally, by \eqref{eq:rhoh-stability} and the inverse estimate
$\norm{q_h}_J\lesssim\norm{q_h}_0$,
\[
 \norm{\tau_h^*}_{\Sigma,\iota}
 +\norm{v_h^*}_J
 \lesssim
 \norm{\sigma_h}_{\Sigma,\iota}
 +\norm{u_h}_J .
\]
Thus, we finish the proof.
\end{proof}

\section{Error estimate}

Throughout this section, the piecewise affine mesh boundary is used as the
boundary representation of the discrete problem, and the boundary data are
evaluated directly on its faces. Any effect of this boundary representation
is incorporated into the formulation and is not separated as an additional
geometric consistency term. Define the strengthened displacement norm
\[
 \|w\|_{J,*}^2
 :=
 \|w\|_0^2+c_h(w,w)
 +\sum_{F\subset\Gamma_f}h_F\|w\|_{0,F}^2.
\]

\begin{lemma}[Consistency and error equation]
\label{lem:consistency-error-equation}
Let $(\sigma,u)$ be a sufficiently regular exact solution satisfying
\[
 \sigma n=g_f\quad\text{on }\Gamma_f,
 \qquad
 u=g_c\quad\text{on }\Gamma_c.
\]
Assume that the boundary data are evaluated exactly in the discrete
right-hand side. Then, for every $(\tau_h,v_h)\in\Sh\times\Uh$,
\begin{equation}\label{eq:consistency}
 B_h^N(\sigma,u;\tau_h,v_h)
 =
 (f,v_h)+\mathcal F_h^N(g_f,g_c,g_\sigma;\tau_h,v_h).
\end{equation}
Consequently,
\begin{equation}\label{eq:galerkin-orthogonality}
 B_h^N(\sigma-\sigma_h,u-u_h;\tau_h,v_h)=0
 \qquad
 \forall(\tau_h,v_h)\in\Sh\times\Uh.
\end{equation}
\end{lemma}

\begin{proof}
Since a general $\tau_h\in\Sh$ does not satisfy
$\tau_h n=0$ on $\Gamma_f$. Integration by parts gives
\begin{align}
 a_\iota(\sigma,\tau_h)+(\diver\tau_h,u)
 &= \langle u,\tau_h n\rangle_{\partial\Omega}
 +\iota^2\langle\mathcal Ag_\sigma,\tau_h\rangle_{\partial\Omega}.
\label{eq:strong-consistency}
\end{align}
Using the definition of $b_h^N$ and $u=g_c$ on $\Gamma_c$, we obtain
\begin{equation}\label{eq:stress-consistency}
 a_\iota(\sigma,\tau_h)+b_h^N(\tau_h,u)
 =
 \langle g_c,\tau_h n\rangle_{\Gamma_c}
 +\iota^2\langle\mathcal Ag_\sigma,\tau_h\rangle_{\partial\Omega}.
\end{equation}
Moreover, $\sigma n=g_f$ on $\Gamma_f$ implies
\begin{equation}\label{eq:penalty-consistency}
 s_h^N(\sigma,\tau_h)=s_h^N(g_f,\tau_h).
\end{equation}
For the equilibrium equation,
\begin{equation}\label{eq:equilibrium-consistency}
 b_h^N(\sigma,v_h)
 =
 (\diver\sigma,v_h)-\langle\sigma n,v_h\rangle_{\Gamma_f}
 =
 (f,v_h)-\langle g_f,v_h\rangle_{\Gamma_f}.
\end{equation}
Since $u$ is single-valued across interior faces and $u=g_c$ on
$\Gamma_c$,
\begin{equation}\label{eq:stabilization-consistency}
 c_h^i(u,v_h)=0,
 \qquad
 c_h(u,v_h)=c_h^c(g_c,v_h).
\end{equation}
Combining \eqref{eq:stress-consistency}--\eqref{eq:stabilization-consistency}
yields \eqref{eq:consistency}. Subtracting the discrete problem
\eqref{lsg-nitsche} gives \eqref{eq:galerkin-orthogonality}.
\end{proof}

\begin{theorem}[Quasi-optimal error estimate]
\label{thm:error-estimate-nitsche}
Let $(\sigma,u)$ be the exact solution and let
$(\sigma_h,u_h)\in\Sh\times\Uh$ solve \eqref{lsg-nitsche}. Then
\begin{equation}\label{eq:quasi-optimal}
 \|\sigma-\sigma_h\|_{\Sigma,\iota}
 +\|u-u_h\|_J
 \lesssim
 \inf_{\tau_h\in\Sh}\|\sigma-\tau_h\|_{\Sigma,\iota}
 +\inf_{v_h\in\Uh}\|u-v_h\|_{J,*}.
\end{equation}
The hidden constant is independent of $h$ and $\iota$.
\end{theorem}

\begin{proof}
Let $\tau_h\in\Sh$ and $v_h\in\Uh$ be arbitrary, and set
\[
 \theta_h:=\sigma_h-\tau_h,
 \qquad
 \xi_h:=u_h-v_h.
\]
By the discrete inf--sup condition,
\begin{align}
 \|\theta_h\|_{\Sigma,\iota}+\|\xi_h\|_J
 &\lesssim
 \sup_{(\eta_h,w_h)\in\Sh\times\Uh\setminus\{(0,0)\}}
 \frac{B_h^N(\theta_h,\xi_h;\eta_h,w_h)}
 {\|\eta_h\|_{\Sigma,\iota}+\|w_h\|_J}.
 \label{eq:infsup-error-start}
\end{align}
Since
\[
 \theta_h=(\sigma-\tau_h)-(\sigma-\sigma_h),
 \qquad
 \xi_h=(u-v_h)-(u-u_h),
\]
Galerkin orthogonality implies
\begin{equation}\label{eq:error-decomposition}
 B_h^N(\theta_h,\xi_h;\eta_h,w_h)
 =
 B_h^N(\sigma-\tau_h,u-v_h;\eta_h,w_h).
\end{equation}
Let
\[
 E_\sigma:=\sigma-\tau_h,
 \qquad
 E_u:=u-v_h.
\]
By the Cauchy--Schwarz inequality, the inverse trace inequality, and the
definitions of the norms,
\begin{equation}\label{eq:continuity-error}
 \left|B_h^N(E_\sigma,E_u;\eta_h,w_h)\right|
 \lesssim
 \left(\|E_\sigma\|_{\Sigma,\iota}+\|E_u\|_{J,*}\right)
 \left(\|\eta_h\|_{\Sigma,\iota}+\|w_h\|_J\right).
\end{equation}
In particular, the two Nitsche boundary terms are controlled by
\[
 \left|\langle E_\sigma n,w_h\rangle_{\Gamma_f}\right|
 \lesssim
 \|E_\sigma\|_{\Sigma,\iota}\|w_h\|_J,
 \qquad
 \left|\langle\eta_h n,E_u\rangle_{\Gamma_f}\right|
 \lesssim
 \|\eta_h\|_{\Sigma,\iota}\|E_u\|_{J,*}.
\]
Therefore,
\begin{equation}\label{eq:discrete-error-bound}
 \|\theta_h\|_{\Sigma,\iota}+\|\xi_h\|_J
 \lesssim
 \|\sigma-\tau_h\|_{\Sigma,\iota}
 +\|u-v_h\|_{J,*}.
\end{equation}
Finally, the triangle inequality implies the
proof.

\end{proof}

\begin{corollary}
\label{cor:first-order}
Assume that
\[
 \sigma\in H^2(\Omega;\Sd),
 \qquad
 u\in H^1(\Omega;\R^d).
\]
Let $I_h\sigma\in\Sh$ be the Scott--Zhang interpolant and let
$P_h^0u\in\Uh$ be the elementwise $L^2$-projection. Then
\[
 \|\sigma-I_h\sigma\|_{\Sigma,\iota}
 \lesssim h\|\sigma\|_{2,\Omega},
 \qquad
 \|u-P_h^0u\|_{J,*}
 \lesssim h|u|_{1,\Omega}.
\]
Consequently,
\begin{equation}\label{eq:first-order-error}
 \|\sigma-\sigma_h\|_{\Sigma,\iota}
 +\|u-u_h\|_J
 \lesssim
 h\left(\|\sigma\|_{2,\Omega}+|u|_{1,\Omega}\right).
\end{equation}
\end{corollary}

\begin{remark}[Boundary representation]
The estimate in Theorem \ref{thm:error-estimate-nitsche} contains no separate geometric consistency term. If the implementation replaces $g_f$, $g_c$, or $g_\sigma$ by additional projections or numerical approximations, the corresponding boundary approximation must be added to the error equation.
\end{remark}

\section{Numerical experiments}
In this section, we present several numerical experiments to verify the
accuracy and robustness of the proposed lowest order numerical scheme. We also investigate
the differences between the linear stress gradient elasticity model and the
classical linear elasticity model in numerical simulations. 
\subsection{Verification}
By \eqref{lsg-strong}, if an exact displacement \(u\) is prescribed, then the corresponding stress \(\sigma\) can only be obtained by solving a Laplace-type equation. This makes it difficult to construct an analytic solution. For this reason, we take an exact solution \(u_0\) of problem \eqref{eq:iota0} as
\[
u_{0}
=
\operatorname{curl}\bigl(\Pi_{i=1}^d\sin^4(\pi x_i)\bigr)
+
\frac{1}{\lambda+1}\left(\Pi_{i=1}^d
x_i^4(1-x_i)^4\right)(\sum_{i=1}^d\mathbf{e}_i).\\
\]
Note that \(u_0\in H_0^2([0,1]^d;\mathbb{R}^d)\). Therefore, for any \(\tau\) satisfying \(\operatorname{div}\tau=0\), we have
\[
(\nabla\mathcal{A}\sigma(u_0),\nabla\tau)
=
(\nabla\operatorname{sym}\operatorname{grad}u_0,\nabla\tau)
=
(u_0,\operatorname{div}\Delta \tau)
=
0.
\]
Hence,
\[
\sigma_{\iota}
=
\sigma_0
:=
\mathcal{A}^{-1}\operatorname{sym}\operatorname{grad}u_0 .
\]
Then 
$
\epsilon(u_{\iota}) = \mathcal{A}\sigma_0 -\iota^2\triangle\mathcal{A}\sigma_0,
$
which implies $\epsilon(u_{\iota})=\epsilon(u_0-\iota^2\triangle u_0)$. And $u_0\in H_0^3(\Omega)$ implies $u_{\iota}= u_0-\iota^2\triangle u_0.$ The parameters setting is $\lambda=1.0e+04,~\mu=0.3,~\iota=0.5,10^{-1},10^{-2},10^{-3}.$ 

\begin{table}[htbp]
\centering
\caption{Numerical experiments for the two-dimensional test problem}
\label{tab:2d-test}
\begin{tabular}{cccccccc}
\hline
hsize & 1/2 & 1/4 & 1/8 & 1/16 & 1/32 & 1/64 & 1/128 \\
\hline
&&&$\iota=0.5$&&&&\\
\hline
$\|E_{\sigma}\|_{\iota}$   
& 6.79e+01 & 6.11e+01 & 4.59e+01 & 2.45e+01 & 1.08e+01 & 4.79e+00 & 2.28e+00 \\
order
& -- & 0.15 & 0.41 & 0.91 & 1.18 & 1.17 & 1.07 \\
$\|E_u\|_{L^2}$
& 4.76e+01 & 4.06e+01 & 2.98e+01 & 1.52e+01 & 6.26e+00 & 2.60e+00 & 1.19e+00 \\
order
& -- & 0.23 & 0.44 & 0.97 & 1.28 & 1.27 & 1.13 \\
\hline
&&&$\iota=0.1$&&&&\\
\hline
$\|E_{\sigma}\|_{\iota}$   
& 5.38e+01 & 3.97e+01 & 2.64e+01 & 1.42e+01 & 7.26e+00 & 3.66e+00 & 1.83e+00 \\
order
& -- & 0.44 & 0.59 & 0.89 & 0.97 & 0.99 & 1.00 \\
$\|E_u\|_{L^2}$
& 3.05e+00 & 2.72e+00 & 1.70e+00 & 8.54e-01 & 4.26e-01 & 2.13e-01 & 1.06e-01 \\
order
& -- & 0.16 & 0.68 & 0.99 & 1.00 & 1.00 & 1.00 \\
\hline
&&&$\iota=1.0e{-2}$&&&&\\
\hline
$\|E_{\sigma}\|_{\iota}$   
& 5.28e+01 & 3.88e+01 & 2.58e+01 & 1.39e+01 & 7.14e+00 & 3.61e+00 & 1.81e+00 \\
order
& -- & 0.44 & 0.59 & 0.89 & 0.96 & 0.98 & 0.99 \\
$\|E_u\|_{L^2}$
& 1.41e+00 & 2.65e+00 & 1.40e+00 & 6.66e-01 & 3.33e-01 & 1.67e-01 & 8.40e-02 \\
order
& -- & -0.91 & 0.92 & 1.07 & 1.00 & 0.99 & 1.00 \\
\hline
&&&$\iota=1.0e{-4}$&&&&\\
\hline
$\|E_{\sigma}\|_{\iota}$   
& 5.28e+01 & 3.88e+01 & 2.58e+01 & 1.39e+01 & 7.13e+00 & 3.61e+00 & 1.81e+00 \\
order
& -- & 0.44 & 0.59 & 0.89 & 0.96 & 0.98 & 0.99 \\
$\|E_u\|_{L^2}$
& 1.40e+00 & 2.65e+00 & 1.40e+00 & 6.64e-01 & 3.32e-01 & 1.67e-01 & 8.37e-02 \\
order
& -- & -0.92 & 0.92 & 1.07 & 1.00 & 0.99 & 0.99 \\
\hline
\end{tabular}
\end{table}

\begin{table}[htbp]
\centering
\caption{Numerical experiments for the three-dimensional verification}
\label{tab:3d-test}
\begin{tabular}{ccccc}
\hline
hsize & 1/2 & 1/4 & 1/8 & 1/16 \\
\hline
&&$\iota=0.5$&&\\
\hline
$\|E_{\sigma}\|_{\iota}$   
& 7.35e+01 & 7.01e+01 & 5.18e+01 & 2.72e+01\\
order
& -- & 0.07 & 0.44 & 0.93  \\
$\|E_{u}\|_{L^2}$   
& 5.01e+01 & 4.60e+01 & 3.30e+01 & 1.65e+01\\
order 
& -- & 0.13 & 0.48 & 1.00  \\
\hline
&&$\iota=0.1$&&\\
\hline
$\|E_{\sigma}\|_{\iota}$   
& 5.77e+01 & 4.51e+01 & 2.72e+01 & 1.44e+01\\
order
& -- & 0.36 & 0.73 & 0.92  \\
$\|E_{u}\|_{L^2}$   
& 2.93e+00 & 2.35e+00 & 1.45e+00 & 7.21e-01\\
order 
& -- & 0.32 & 0.70 & 1.01  \\
\hline
&&$\iota=1.0e-02$&&\\
\hline
$\|E_{\sigma}\|_{\iota}$   
& 5.66e+01 & 4.42e+01 & 2.66e+01 & 1.41e+01\\
order
& -- & 0.36 & 0.73 & 0.92  \\
$\|E_{u}\|_{L^2}$   
& 1.23e+00 & 2.29e+00 & 1.16e+00 & 5.42e-01\\
order 
& -- & -0.90 & 0.99 & 1.09  \\
\hline
&&$\iota=1.0e-04$&&\\
\hline
$\|E_{\sigma}\|_{\iota}$   
& 5.66e+01 & 4.42e+01 & 2.66e+01 & 1.41e+01\\
order
& -- & 0.36 & 0.73 & 0.92  \\
$\|E_{u}\|_{L^2}$   
& 1.22e+00 & 2.30e+00 & 1.15e+00 & 5.40e-01\\
order 
& -- & -0.91 & 0.99 & 1.10  \\
\hline
\end{tabular}
\end{table}
Tab. \ref{tab:2d-test} and \ref{tab:3d-test} show that the convergence rate  almost equal to $1$ for both $\sigma$ and $u$. This implies the scheme is robust with $\lambda$ and $\iota$. 

\subsection{Plate with a rounded U-shaped notch}
In this section, we consider a plate with a U-notched hole, see Fig. \ref{fig:U-notched}.
The rounded notch induces a pronounced geometric stress concentration
near its upper root. The hole extends a distance $r$ along the $x$-axis and
$2r$ along the $y$-axis. 
The radius $r$ characterizes the size of the hole, and is therefore used as the
reference length for defining the dimensionless internal length
$\iota/r$.
\begin{figure}[htbp]
  \centering
  \includegraphics[width=0.45\linewidth]{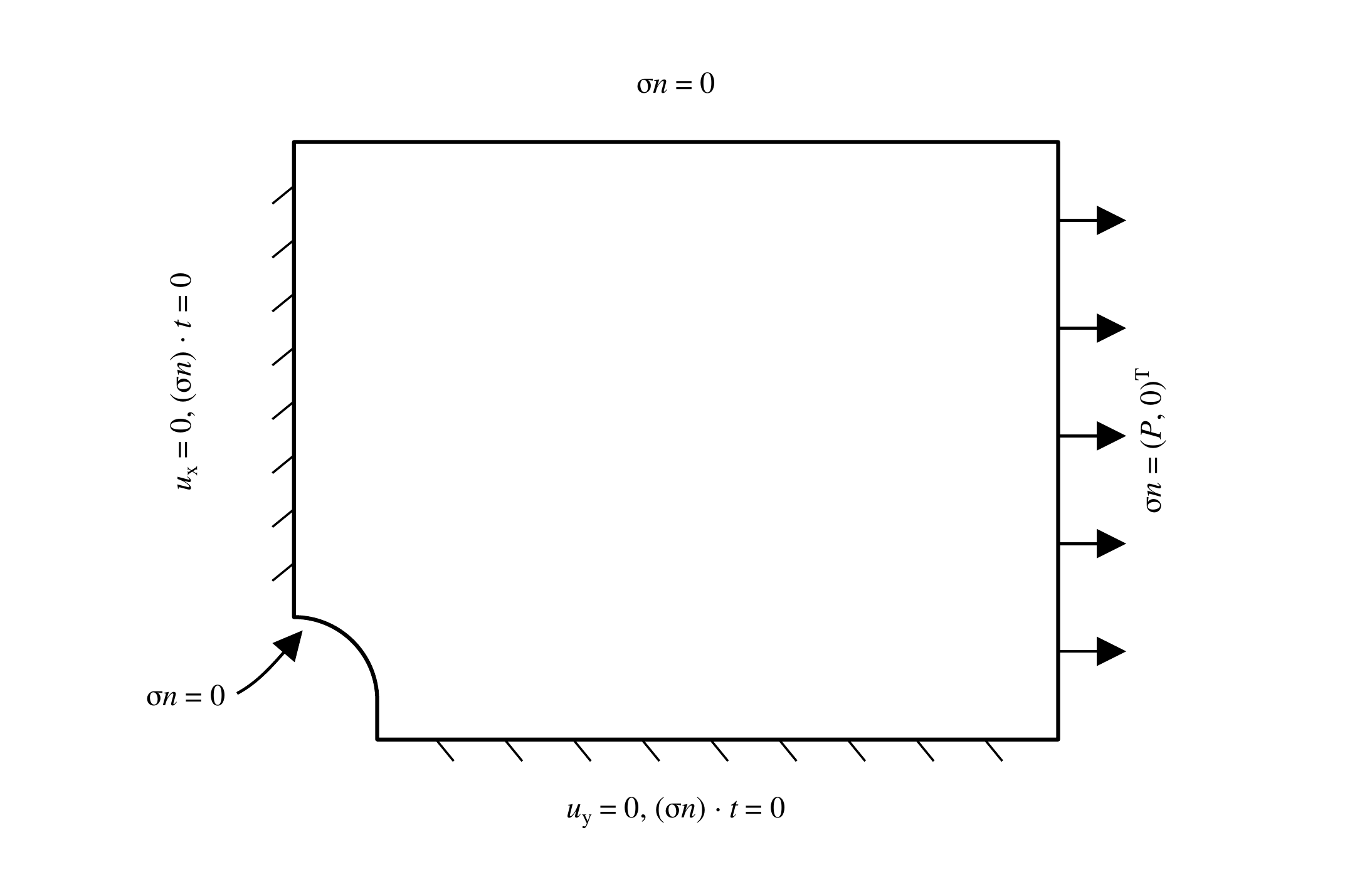}
  \includegraphics[width=0.50\linewidth]{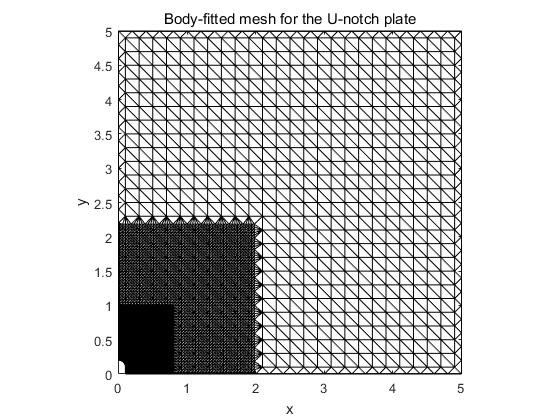}
  \caption{Configure of U-notched hole(left) and mesh(right).}
  \label{fig:U-notched}
\end{figure}
Owing to symmetry, only one quarter of the plate is discretized.
the notch boundary in this experiment is given by\[
\Gamma_{\mathrm{notch}}
=
\bigl\{(r,y)\mid 0\leq y\leq r\bigr\}
\cup
\bigl\{(x,y)\in\overline{\Omega}\mid
x^2+(y-r)^2=r^2,\ r\leq y\leq 2r\bigr\}.
\]
The upper notch root is located at
\[
a_0=(0,2r).
\]
We set $L=H=5$, $r=0.1$, $E=1000$, and $\nu=0.3$, and prescribe a uniform tensile traction $P=1000$ on the right boundary. The upper boundary and the notch boundary are traction free, while the standard symmetry conditions are imposed on $x=0$ and $y=0$. No body force is applied. And set
$
\partial_n\sigma=0
\text{ on }\partial\Omega.
$
The problem is solved using the Nitsche-stabilized
method \eqref{lsg-nitsche} on a body-fitted graded mesh, see Fig. \ref{fig:U-notched} , with
\[
h_{\rm near}=\frac{r}{24},
\qquad
h_{\rm far}=0.1.
\]
The geometry and the mesh are kept fixed, whereas the dimensionless internal length is varied over
\[
\frac{\iota}{r}
\in
\{0,\,0.01,\,0.05,\,0.1,\,0.5,\,1,\,5,\,10\},
\]
where $\iota/r=0$ corresponds to the classical linear elasticity model. 
Define
\[
K^{a}_t:=\frac{\sigma_{xx}(a)}{P}.
\]
Specially, $K^{\mathrm{root}}_t$ represent the factors on root point $a_0$. And define 
\[
K_t^{\mathrm{local}}:=\max_{|a-a_0|<2r}\{K_t^{a}\}.
\]
Fig. \ref{fig:Kt} shows the normalized tensile-stress profile ahead of the notch, the stress-concentration factors evaluated at the notch root and within a neighborhood of the root. Fig. \ref{fig:VM_distribution} gives the stress distribution when $\iota/r=0$ and $\iota/r=10.$ 
\begin{figure}[htbp]
    \centering
    \includegraphics[width=0.4\linewidth]{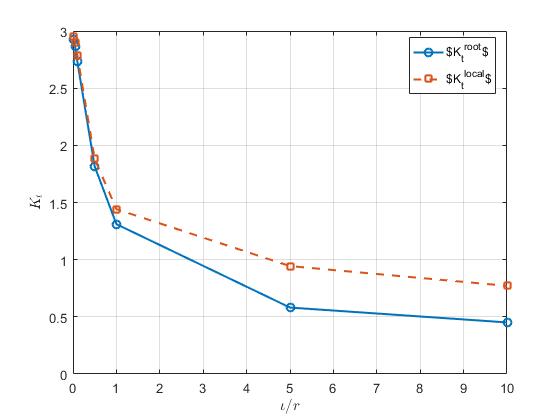}
    \includegraphics[width=0.4\linewidth]{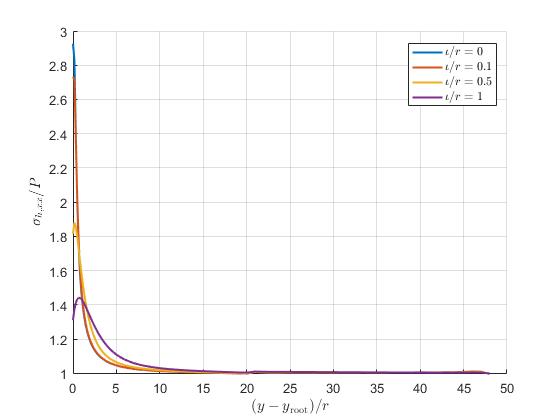}
    \caption{Stress-concentration factors $K_t^{\rm root}$ and
$K_t^{\rm local}$ (left), and normalized tensile-stress profiles along
the line $x=0$ (right), for varying $\iota/r$.}
    \label{fig:Kt}
\end{figure}
\begin{figure}[htbp]
    \centering
    \includegraphics[width=0.45\linewidth]{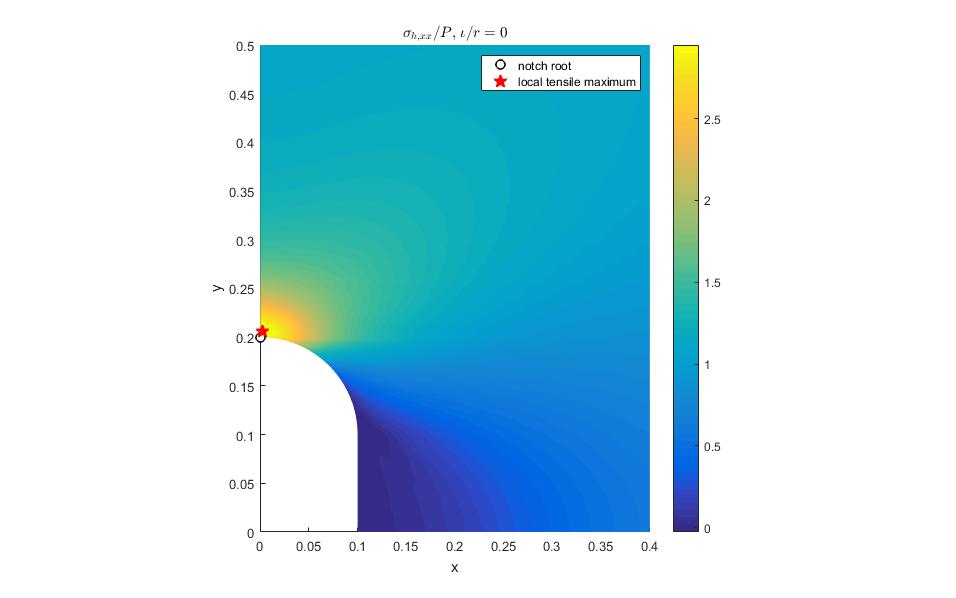}
    \includegraphics[width=0.45\linewidth]{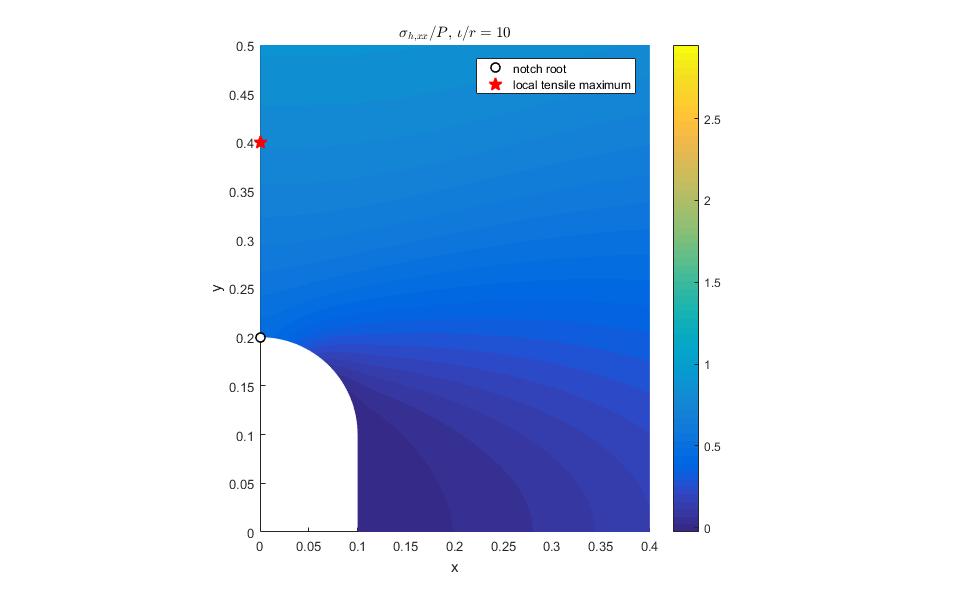}
    \caption{ Stress distribution for $\iota/r=0$(left) and $\iota/r=10$(right).}
    \label{fig:VM_distribution}
\end{figure}

In the classical-elasticity reference computation $\iota/r=0$, the
root value and the local maximum are approximately
$K_t^{\rm root}=2.93$ and $K_t^{\rm local}=2.95$, respectively.
Numerically shows that the local maximum point is almost the root point. For
$\iota/r\ll1$, the stress field remains close to its classical
counterpart. Size effects become pronounced when the internal length
is comparable to the characteristic notch size, as illustrated by
the cases $\iota/r=0.5$ and $1$. The stress-gradient regularization
smooths the stress field and substantially reduces the local peak.
For $\iota/r=1$, the computed peak is attained slightly ahead of the
geometric root. When $\iota/r=5$ and $10$, the internal length is much
larger than the notch-root radius, and the maximum axial stress within
the selected notch neighborhood falls below the applied nominal
traction.

\bibliographystyle{plain}
\bibliography{ref}
\end{document}